\documentclass{amsart}
\usepackage{amsfonts}
\usepackage{latexsym}
\usepackage{amssymb}
\usepackage{amsmath}

\newcommand{\R}{\mathbb R}

\newcommand{\E}{\mathbb E}
\newcommand{\Pro}{\mathbb P}

\newtheorem{thm}{Theorem}[section]
\newtheorem{cor}{Corollary}[section]
\newtheorem{lemma}{Lemma}[section]

\newtheorem{proposition}{Proposition}[section]

\theoremstyle{remark}
\newtheorem*{rmk}{Remark}

\begin{document}


\title{Mean width of Random Perturbations of Random Polytopes}


\author{David Alonso-Guti\'{e}rrez}
\address{Departamento de Matem\'aticas, Universidad de Murcia, Campus de
Espinar\-do, 30100-Murcia, Spain} \email{davidalonso@um.es}

\author{Joscha Prochno}
\address{Institute of Analysis, Johannes Kepler University Linz,
Altenbergerstrasse 69, 4040 Linz, Austria} \email{joscha.prochno@jku.at}


\thanks{The first author is partially supported by MICINN project MTM2010-16679,
MICINN-FEDER project MTM2009-10418 and ``Programa de Ayudas a Grupos de
Excelencia de la Regi\'on de Murcia'', Fundaci\'on S\'eneca,
04540/GERM/06. The second author is supported by the Austrian Science Fund,
FWF project P23987 ``Projection operators in
Analysis and geometry of classical Banach spaces''.}

\date{\today}

\begin{abstract}
We prove some ``high probability'' results on the expected value of the mean width for random perturbations of random polytopes. The random perturbations are considered for Gaussian and $p$-stable random vectors, as well as uniform distributions on $\ell_p^N$-balls and the unit sphere. \end{abstract}

\maketitle

\section{Introduction and notation}

The convex hull of $N$ independent random points is called a random polytope. Their study was initiated by Sylvester with a problem posed in the April issue of The Educational Times in 1864 \cite{S}. He asked for the probability that four points chosen uniformly at random in an indefinite plane have a convex hull which is a four-sided polygon. Within a year it was understood that Sylvester's question was ill-posed. Therefore, he modified the question, asking for the probability that four points chosen independently and uniformly at random from a convex set $K$ in the plane form a four-sided polygon. The problem became known as the famous ``four-point problem'' and was the starting point of extensive research (see also \cite{B} and the references therein).

Later, in their seminal papers \cite{RS1}, \cite{RS2}, \cite{RS3} R\'enyi and Sulanke focussed their investigations on the asymptotic behavior of the expected volume of a random polytope as the number of points $N$ tends to infinity.

Since then, random polytopes received increasing attention, especially in the last decades. Among other things, important quantities are expectations, variances, and distributions of geometric functionals associated to the random polytope. Examples are the volume, the number of vertices, intrinsic volumes, the distance between the random polytope and $K$, and the mean width, just to mention a few.

The study is also stimulated by important applications and connections to various other fields. Those can be found not only in statistics in the form of extreme points of random samples or in convex geometry used to approximate convex sets, but also in theoretical computer science in analyzing the average complexity of algorithms. In view of random perturbations, in this context it is important to mention the groundbreaking work \cite{ST} by Spielman and Teng in which they introduced the concept of ``smooth analysis'' which is a finer concept than worst-case or average-case analysis, using small random perturbations of worst-case inputs of the algorithm. This was crucial to understand the excellent  performance of the simplex method, allowing them to show that it has polynomial ``smoothed complexity''.

In this paper we are interested in one of the aforementioned geometric functionals, namely the expected value of the mean width. We study this functional for randomly perturbed random polytopes and give ``high probability'' estimates for several types of perturbations such as Gaussian, $p$-stable, uniform distributions on the unit sphere and on $\ell_p^N$-balls. Crucial in the proofs of the main results is the so-called ``concentration of measure phenomenon'', going back to an idea of L\'evy, and pushed forward and emphasized by V. Milman in the 1970's in his work on asymptotic geometric analysis (see also \cite{MS}, \cite{L}). Another important tool is the central limit theorem for isotropic log-concave random vectors \cite{K}, which shows that for many directions, the density of the $1$-dimensional marginals of an isotropic log-concave random vector is approximately Gaussian in some range.

Now, let us introduce the notation we need in order to state our results. A log-concave random vector $X$ in $\R^n$ is a random vector whose density with respect to the Lebesgue measure is $f(x)=e^{-V(x)}$, with $V:\R^n\to (-\infty,\infty]$ a convex function. A log-concave random vector is said to be isotropic if it is centered and its covariance matrix is the identity:
\begin{itemize}
\item $\E X=0$
\item$\E X_iX_j=\delta_{i,j}$,
\end{itemize}
where $\E$ denotes the expectation and $\delta_{i,j}$ is the Kronecker delta. We will denote by $\Pro_X$ and $\E_X$ the probability and expectation with respect to the random vector $X$, or simply $\Pro$, and $\E$ when no confusion is possible. If $\theta$ is a vector in the Euclidean unit sphere $S^{n-1}$, $f_\theta$ will denote the density of the $1$-dimensional marginal $\langle X,\theta\rangle$, where $\langle\cdot,\cdot\rangle$ denotes the usual scalar product in $\R^n$.

Examples of isotropic log-concave random vectors are standard Gaussian random vectors or random vectors uniformly distributed in $\frac{K}{L_K}$, where $K$ is an isotropic convex body and $L_K$ is its isotropic constant.

Let $X_1,\dots,X_N$ be independent copies of an isotropic log-concave random vector $X$ in $\R^n$ $(n\leq N$). The random polytope $K_N$ will be defined as their symmetric convex hull, i.e.,
$$
K_N:=\textrm{conv}\{\pm X_1,\dots,\pm X_N\}.
$$
If $y$ is a vector in $\R^N$, the perturbation of $K_N$ given by $y$ will be denoted by $K_{N,y}$ and is defined by
$$
K_{N,y}:=\textrm{conv}\{\pm y_1X_1,\dots,\pm y_NX_N\}.
$$

If $K\subseteq\R^n$ is a convex body the support function of $K$ is defined by
$$
h_K(x):=\max\{\langle x,y\rangle \,:\,y\in K\}.
$$
The mean width of $K$ is
$$
w(K):=\int_{S^{n-1}}h_K(\theta) \, d\sigma(\theta),
$$
where $d\sigma$ is the uniform Haar probability measure on $S^{n-1}$.

Our first result involves perturbations of a random polytope when the perturbation is a standard Gaussian random vector:

\begin{thm}\label{THM_Gaussian_perturbation}
Let $X_1,\dots,X_N$  be independent copies of an isotropic log-concave random vector in $\R^n$ ($n\leq N\leq e^{\sqrt n}$) and let $G$ be a Gaussian random vector in $\R^N$. Then there exist absolute constants $c,c_1,c_2$ such that for every $t>0$
$$
\Pro_G\left(c_1(1-t)\leq\frac{\E_{X_1,\dots,X_N}w(K_{N,G})}{\log N}\leq c_2(1+t)\right)\geq1-\frac{1}{N^{c t^2}}.
$$
\end{thm}

In our second result we also consider a random perturbation of a random polytope where the random perturbation does not have independent coordinates:
\begin{thm}\label{THM_Sphere_perturbation}
Let $X_1,\dots,X_N$  be independent copies of an isotropic log-concave random vector in $\R^n$ ($n\leq N\leq e^{\sqrt n}$) and let $u$ be a random vector uniformly distributed on $S^{N-1}$. Then there exist  absolute constants $c,c_1,c_2$ such that for every $t>0$
$$
\sigma\left(u\in S^{N-1}\,:\,c_1(1-t)\leq\frac{\E_{X_1,\dots,X_N}w(K_{N,u})}{\frac{\log N}{\sqrt N}}\leq c_2(1+t)\right)\geq1-\frac{1}{N^{c t^2}}.
$$
\end{thm}

Another random perturbation we consider, with non-independent coordinates, is the case in which the vector giving the perturbation is uniformly distributed in the unit ball of $\ell_p^N$, which we denote by $B_p^N$. We will prove the following:

\begin{thm}\label{THM_B_p_perturbation}
Let $X_1,\dots,X_N$  be independent copies of an isotropic log-concave random vector in $\R^n$ ($n\leq N\leq e^{\sqrt n}$) and let $y$ be a random vector uniformly distributed in $B_p^N$. Then there exist absolute constants, $c_1,c_2,c,c^\prime$ such that for every $t>0$
$$
\Pro_y\left(c_1(1-t)\leq\frac{\E_{X_1,\dots,X_N}w(K_{N,y})}{\frac{(\log N)^{\frac{1}{p}+\frac{1}{2}}}{N^\frac{1}{p}}}\leq c_2(1+t)\right)\geq1-\frac{1}{N^{\frac{(c t)^p}{p}}}.
$$
\end{thm}

Finally, we will consider random perturbations of random polytopes, where the perturbation is given by a random vector whose coordinates are independent identically distributed $p$-stable random variables. A real valued random variable $\xi$ is called normalized symmetric $p$-stable for some $p\in(0,2]$ if its characteristic function is
$$
\mathbb \phi_{\xi}(x)= \mathbb E e^{i\xi x} = e^{-|x|^p}.
$$
In the case that $1<p<2$ we have finite first moments, but no finite variance. Gaussian random variables are $2$-stable. However, their behavior is rather different from the one of $p$-stable random variables when $p$ is close to $2$. We will prove the following:

\begin{thm}\label{THM_p-stable_perturbation}
Let $X_1,\dots,X_N$  be independent copies of an isotropic log-concave random vector in $\R^n$ and let $\xi=(\xi_1,\ldots,\xi_N)$ where $\xi_1,\ldots,\xi_N$ are independent identically distributed symmetric normalized $p$-stable random variables with $\frac 3 2 <p <2$. Then there exist absolute constants $c,c_1,c_2,C>0$ such that
$$
\mathbb P_\xi \left( c_1(1-t) \leq \frac{\mathbb \E_{X_1,\dots,X_N} \omega(K_{N,\xi})}{N^{\frac{1}{p}}}  \leq c_2(1+t) \right) \geq 1 - C \frac{\sqrt{\log N}}{N^{\frac{1}{p}}}
$$
for every $t$ with
$t^p \geq cM \frac{\sqrt{\log N}^p}{N} \log(M)\log(1+2M\log(M))$, $M=\frac{1}{2-p}$.
\end{thm}

The paper is organized as follows: in Section \ref{Preliminaries} we will state the main tools we will use in our proofs. In Section \ref{RandomPerturbations} we will prove the aforementioned theorems. Finally, in Section \ref{ArbitraryPerturbations} we will prove an estimate for the mean width of some arbitrary perturbations of a random polytope.

Given an isotropic log-concave random vector $X$ and any $1\leq p<\infty$, its $L_p$-centroid body is defined via its support function
$$
h_{Z_p(X)}(\theta)=\left(\E|\langle X,\theta\rangle|^p\right)^\frac{1}{p}.
$$
We will use the notation $a\sim b$ to indicate the existence of two positive absolute constants $c_1,c_2$ such that $c_1a\leq b\leq c_2a$. $c,c^\prime,c_1,c_2,C,\dots$ will always denote positive absolute constants whose value may change from line to line. Throughout this paper, $|\cdot|$ will denote the Lebesgue volume, the absolute value as well as the Euclidean norm and the meaning will be clear from the context.

\section{Preliminary Results}\label{Preliminaries}

A convex function $M:[0,\infty)\to[0,\infty)$ where $M(0)=0$ and $M(t)>0$ for $t>0$ is called an Orlicz function. The conjugate function or dual function $M^*$ of an Orlicz function $M$ is given by the
Legendre transform
  $$
    M^*(x) = \sup_{t\in[0,\infty)}(xt-M(t)).
  $$
Again, $M^*$ is an Orlicz function and $M^{**}=M$. For instance, taking $M(t)=\frac{1}{p}t^p$, $p\geq 1$, the dual function is given by $M^*(t)=\frac{1}{p^*}t^{p^*}$ with $\frac{1}{p^*}+\frac{1}{p}=1$.
The $n$-dimensional {Orlicz space} $\ell_M^n$ is $\R^n$ equipped with the norm
  $$
    \Vert{x}\Vert_M = \inf \left\{ \rho>0 \,:\, \sum_{i=1}^n M\left(\tfrac{|x_i|}{\rho}\right) \leq 1 \right\}.
  $$
In case $M(t)=t^p$, $1\leq p<\infty$, we just have $\Vert\cdot\Vert_M=\Vert\cdot\Vert_p$. For a detailed and thorough introduction to the theory of Orlicz spaces we refer the reader to \cite{key-KR} and \cite{key-RR}.\\

In \cite{GLSW2} the authors obtained the following result:

\begin{thm}[\cite{GLSW2} Lemma 5.2]\label{THM Schuett Werner Litvak Gordon}
Let $X_1,\ldots,X_N$ be independent identically distributed random variables with finite first moments. For all $s\geq 0$ let
$$
  M(s)=\int_0^s\int_{\{\frac{1}{t}\leq |X_1|\}} |X_1| \,d\mathbb P \, dt.
$$
Then, for all $x=(x_i)_{i=1}^N\in\R^N$,
$$
\mathbb E \max\limits_{1\leq i \leq N}|x_iX_i|\sim \Vert x\Vert_M.
$$
\end{thm}

Obviously, the function
\begin{equation}\label{EQU Orlicz function M}
  M(s)=\int_0^s \int_{\{\frac{1}{t}\leq |X_1|\}} |X_1| \, d\mathbb P \, dt
\end{equation}
is non-negative and convex, since $\int_{\{\frac{1}{t}\leq |X|\}}|X| \,d\mathbb P$ is increasing in $t$. Furthermore, we have
$M(0)=0$ and $M$ is continuous.

As a corollary we obtain the following result:

\begin{cor}\label{CorollaryOrlicz}
Let $X_1,\ldots,X_N$ be independent identically distributed random vectors in $\R^n$ and let $K_N=\textrm{conv}\{\pm X_1,\ldots,\pm X_N\}$. Let $\theta\in S^{n-1}$ and
$$
  M_{\theta}(s)=\int_0^s \int_{\{\frac{1}{t}\leq |\langle X_1,\theta\rangle|\}} |\langle X_1,\theta\rangle| \, d\mathbb P \, dt.
$$
Then, for every $y\in\R^N$,
$$
\E h_{K_{N,y}}(\theta)\sim\Vert y\Vert_{M_\theta}.
$$
\end{cor}

Another important tool we will use is Klartag's central limit theorem from \cite{K}. Here, $\gamma$ stands for the density of the standard Gaussian, i.e., $\gamma(t) = \frac{1}{\sqrt{2\pi}}e^{-\frac{t^2}{2}}$.

\begin{thm}[\cite{K}, Theorem 1.4]\label{THM_CLT}
  Let $n\geq 1$ be an integer and let $X$ be an isotropic, log-concave, random vector in $\R^n$. Then there exists $\Theta\subseteq S^{n-1}$ with $\sigma_{n-1}(\Theta) \geq 1-Ce^{-\sqrt{n}}$ such that for all
  $\theta \in \Theta$, the real valued random variable $\langle X,\theta \rangle$ has a density $f_{\theta}:\R^n \to [0,\infty)$ with the following properties:
  \begin{enumerate}
  \item $\int_{-\infty}^{\infty} |f_{\theta}(t) - \gamma(t)| \, dt \leq \frac{1}{n^{\kappa}}$,
  \item For all $|t|\leq n^{\kappa}$ we have $|\frac{f_{\theta}(t)}{\gamma(t)} - 1| \leq \frac{1}{n^{\kappa}}$.
  \end{enumerate}
  Here, $C,\kappa>0$ are universal constants.
\end{thm}

Finally, in order to prove that the estimates of the expected mean width of a random perturbation of a random polytope hold with high probability, we will need some concentration of measure results for the random vector that defines the perturbation.

The concentration of measure inequality on the sphere states the following:

\begin{thm}\label{THM_Sphere_concentration}
There exist absolute constants $c,C$ such that if $f:\R^N\to\R$ is 1-Lipschitz, then for all $t>0$
$$
\sigma\left(\theta\in S^{N-1}\,:\,|f(\theta)-\mathbb E f(\theta)| \geq t\right)\leq Ce^{-ct^2N}.
$$
\end{thm}

As a consequence, since the Gaussian measure is rotationally invariant, we have the following:

\begin{thm}\label{THM_Gaussian_concentration}
There exist absolute constants $c,C$ such that if $f:\R^N\to\R$ is 1-Lipschitz and $G$ is a standard Gaussian random vector in $\R^N$, then for all $t>0$
$$
\Pro_G\left(|f(G)-\mathbb E f(G)| \geq t\right)\leq Ce^{-ct^2}.
$$
\end{thm}

In \cite{SZ2}, the following concentration of measure result on $B_p^N$ ($1\leq p\leq 2$) was proved. The case $2\leq p<\infty$, for $0<t<2$ is a consequence of the concentration of measure theorem in uniformly convex spaces proved in \cite{GM}:

\begin{thm}\label{THM_B_p_concentration}
There exist absolute constants $c,C$ such that if $f:B_p^N\to\R$ is 1-Lipschitz and $y$ is a random vector uniformly distributed on $B_p^N$ ($1\leq p\leq2$), then for all $t>0$
$$
\Pro_y\left(|f(y)-\mathbb E f(y)| \geq t\right)\leq Ce^{-\frac{c^pt^pN}{p}}.
$$
\end{thm}

The corresponding concentration of measure result for $p$-stable random vectors can be found in \cite{HM}, which we state here not in its most general form:

\begin{thm} \label{THM_p_stable_concentration}
Let $y=(\xi_1,\dots,\xi_N)$ be a random vector, with $\xi_1,\dots\xi_N$ independent symmetric normalized $p$-stable random variables, $p >\frac 3 2$, and let $M=\frac{1}{2-p}$. Let $f:\R^N \to \R$ be a 1-Lipschitz function. Then there exists a constant $C>0$ such that
$$
\mathbb P_y \left( |f(y)-\mathbb E f(y)| \geq t\right) \leq C \frac{1}{t^p},
$$
for all $t$ with $t^p \geq 4 M \log(M) \log(1+2M\log(M))$.
\end{thm}

\section{Random Perturbations of Random Polytopes}\label{RandomPerturbations}

In this section we will prove Theorems \ref{THM_Gaussian_perturbation},\ref{THM_Sphere_perturbation} \ref{THM_B_p_perturbation} and \ref{THM_p-stable_perturbation}. The proofs of the three results follow the same lines: we consider the function $f:\R^N\to\R$
$$
f(y)=\E_{X_1,\dots,X_N}w(K_{N,y})
$$
and we apply the concentration of measure theorems to $\frac{f}{L}$, where $L$ is the Lipschitz constant of $f$. In order to do that we need to compute the expectation of $f$ in the three cases and the value of $L$. In the following lemma we compute the value of $L$.

\begin{lemma}\label{LemmaLipsichtzConstant}
Let $X_1,\dots, X_N$ ($n\leq N\leq e^{\sqrt n}$) be independent copies of an isotropic log-concave random vector in $\R^n$, and let $f:\R^N\to\R$ be the function $f(y)=\E_{X_1,\dots,X_N}w(K_{N,y})$. Then there exists an absolute constant $C$ such that for any $y_1,y_2\in\R^N$ we have
$$
|f(y_1)-f(y_2)|\leq C\sqrt{\log N}|y_1-y_2|.
$$
\end{lemma}
\begin{proof}
Let $y_1,y_2\in\R^N$ be any two vectors
\begin{eqnarray*}
|f(y_1)-f(y_2)|&=&\left|\int_{S^{n-1}}\E_{X_1,\dots,X_N}\left(\max_{1\leq i\leq N}|\langle y_1(i)X_i,\theta\rangle|-\max_{1\leq i\leq N}|\langle y_2(i)X_i,\theta\rangle|\right) \, d\sigma(\theta)\right|\cr
&\leq&\int_{S^{n-1}}\E_{X_1,\dots,X_N}\left|\max_{1\leq i\leq N}|\langle y_1(i)X_i,\theta\rangle|-\max_{1\leq i\leq N}|\langle y_2(i)X_i,\theta\rangle|\right| \,d\sigma(\theta)\cr
&\leq&\int_{S^{n-1}}\E_{X_1,\dots,X_N}\max_{1\leq i\leq N}|\langle (y_1(i)-y_2(i))X_i,\theta\rangle| \,d\sigma(\theta)\cr
&\leq&\max_{u\in S^{N-1}}\int_{S^{n-1}}\E_{X_1,\dots,X_N}\max_{1\leq i\leq N}|\langle u_iX_i,\theta\rangle| \, d\sigma(\theta)|y_1-y_2|\cr
&=&\max_{u\in S^{N-1}}\E_{X_1,\dots,X_N}w(K_{N,u})|y_1-y_2|\cr
&\leq&\E_{X_1,\dots,X_N}w(K_{N})|y_1-y_2|,\cr
\end{eqnarray*}
since $|u_i|\leq 1$ for all $1\leq i\leq N$. Since $\E_{X_1,\dots,X_N}w(K_{N})\leq C\sqrt{\log N}$ the result follows (see \cite{DGT2} for a proof in the context of random polytopes in isotropic convex bodies). For the sake of completeness we provide here a proof of this fact in the general context of isotropic log-concave vectors:
take $p=\log N$
\begin{eqnarray*}
\E_{X_1,\dots,X_N}w(K_{N})&=&\int_{S^{n-1}}\E_{X_1,\dots,X_N}\max_{1\leq i\leq N}|\langle X_i,\theta\rangle| \, d\sigma(\theta)\cr
&\leq &C\int_{S^{n-1}}\E_{X_1,\dots,X_N}\left(\sum_{i=1}^N|\langle X_i,\theta\rangle|^p\right)^\frac{1}{p} \, d\sigma(\theta)\cr
&\leq &C\int_{S^{n-1}}\left(\E_{X_1,\dots,X_N}\sum_{i=1}^N|\langle X_i,\theta\rangle|^p\right)^\frac{1}{p} \, d\sigma(\theta)\cr
&=& CN^\frac{1}{p}\int_{S^{N-1}}\left(\E_{X_1}|\langle X_1,\theta\rangle|^p\right)^\frac{1}{p} \, d\sigma(\theta)\cr
&=& Cw(Z_{\log N}(X)).
\end{eqnarray*}
Now $w(Z_{\log N}(X))\sim\sqrt{\log N}$ whenever $1\leq N\leq e^{\sqrt n}$ (see \cite{Pa} for a proof in the context of random vectors in an isotropic convex body. The same proof works for a general isotropic log-concave random vector).
\end{proof}

The following lemma was essentially proved in \cite{AGP2}. Since it will be crucial in order to estimate $\E f(y)$ when $y$ is a random perturbation, we include the proof here. It is based on Klartag's central limit theorem.

\begin{lemma}\label{ProbEstimateCLT}
Let $X$ be an isotropic log-concave random vector in $\R^n$. If $n\leq N\leq n^\delta$ there exists a set of directions $\Theta\subseteq S^{n-1}$ with $\sigma(\Theta)\geq 1-Ce^{-\sqrt n}$ and a constant $\kappa$ such that if $\theta\in\Theta$ and $\alpha^2 < \frac{\kappa}{\delta}$ then
$$
\Pro\left(| \langle X,\theta\rangle| \geq \alpha\sqrt{\log N}\right) \geq\frac{c}{N^{\frac{\alpha^2}{2}}\sqrt{\log N}}\,.
$$
\end{lemma}

\begin{proof}
By Theorem \ref{THM_CLT} there exist a constant $\kappa$ and a set $\Theta\subseteq S^{n-1}$ with
measure greater than $1-Ce^{-\sqrt n}$ such that for any $\theta\in
\Theta$ and any $0\leq t\leq cn^{\kappa}$ we have, using the well known estimate for every $t\geq 1$
$$
\frac{\gamma(t)}{2t}\leq\int_t^\infty\gamma(s) \, ds\leq\frac{2\gamma(t)}{t},
$$
\begin{eqnarray*}
\mathbb P \left(|\langle X,\theta\rangle|<t\right) & \leq & \left(1+\frac{C^\prime}{n^\kappa}\right)\left(1-2\int_t^\infty\gamma(s) \, ds\right) \\
& \leq & \left(1+\frac{C^\prime}{n^{\kappa}}\right)\left(1-\frac{\gamma(t)}{t}\right),
\end{eqnarray*}
and so
$$
\mathbb P \left( |\langle X,\theta\rangle|\geq t\right)\geq 1-\left(1+\frac{C^\prime}{n^{\kappa}}\right)\left(1-\frac{\gamma(t)}{t}\right).
$$
Taking $t=\alpha\sqrt{\log N}$ we have that $t\leq cn^{\kappa}$,
since $N\leq n^\delta$. Thus
\begin{eqnarray*}
\mathbb P \left( |\langle X,\theta\rangle| \geq \alpha\sqrt{\log N}\right)
&\geq& \frac{e^{-\frac{\alpha^2}{2}\log
    N}}{\sqrt{2\pi}\alpha\sqrt{\log
    N}}-\frac{C^\prime}{n^{\kappa}}\left(1-\frac{e^{-\frac{\alpha^2}{2}\log
      N}}{\sqrt{2\pi}\alpha\sqrt{\log N}}\right)\cr
&\geq &\frac{1}{\sqrt{2\pi}\alpha N^{\frac{\alpha^2}{2}}\sqrt{\log
    N}}-\frac{C^\prime}{n^{\kappa}}\cr
&\geq &\frac{1}{\sqrt{2\pi}\alpha N^{\frac{\alpha^2}{2}}\sqrt{\log
    N}}-\frac{C^\prime}{N^{\frac{\kappa}{\delta}}}\cr
&\geq&\frac{1}{\sqrt{2\pi}\alpha N^{\frac{\alpha^2}{2}}\sqrt{\log
    N}}\left(1-\frac{C^{\prime\prime}\alpha\sqrt{\log N}}{N^{\frac{\kappa}{\delta}-\frac{\alpha^2}{2}}}\right)\cr
    &\geq&\frac{C}{\sqrt{2\pi}\alpha N^{\frac{\alpha^2}{2}}\sqrt{\log
    N}},
\end{eqnarray*}
whenever $N\geq N_0$ if we take $\frac{\alpha^2}{2}<\frac{\kappa}{\delta}$.
\end{proof}
In the following lemmas we will compute the expected value of $f(y)$, when $y$ is distributed according to the previously mentioned distributions. Since the techniques we use to compute it are quite different when the number of vertices is big and when the number of vertices is small, we divide both cases into separate lemmas.

\begin{lemma}\label{LemmaExpectationGaussianFewVertices}
Let $X_1,\dots, X_N$ ($n\leq N\leq n^\delta$) be independent copies of an isotropic log-concave random vector in $\R^n$, and let $f:\R^N\to\R$ be the function $f(y)=\E_{X_1,\dots,X_N}w(K_{N,y})$. There exists a constant depending on $\delta$, $c_1(\delta)$ and an absolute constant $c_2$ such that if $G=(g_1,\dots,g_N)$ is a standard Gaussian random vector in $\R^N$ then
$$
c_1(\delta)\log N\leq\E f(G)\leq c_2\log N.
$$
\end{lemma}
\begin{proof}
For any $\theta\in S^{N-1}$ we have that
\begin{eqnarray*}
\E_G\E_{X_1,\dots,X_N}h_{K_{N,G}}(\theta)&=&\E_G\E_{X_1,\dots, X_N}\max_{1\leq i\leq N}|\langle g_iX_i,\theta\rangle|\cr
&\leq&\E_G\max_{1\leq i\leq N}|g_i|\E_{X_1,\dots, X_N}\max_{1\leq i\leq N}|\langle X_i,\theta\rangle|\cr
&\leq&\sqrt{\log N}\E_{X_1,\dots,X_N} h_{K_N}(\theta).
\end{eqnarray*}
Integrating in $\theta$ we obtain the upper bound
$$
\E f(G)\leq c_2\log N.
$$
On the other hand, for any $\theta\in S^{n-1}$ we have
\begin{eqnarray*}
\E_G\E_{X_1,\dots X_N}h_{K_{N,G}}(\theta)&=&\E_G\E_{X_1,\dots X_N}\max_{1\leq i\leq N}|g_i||\langle X_i,\theta\rangle|\sim\Vert (1,\dots,1)\Vert_{N_\theta}\cr
&=&\inf\left\{s>0\,:\,N_\theta\left(\frac{1}{s}\right)\leq\frac{1}{N}\right\}
\end{eqnarray*}
where $N_\theta$ is the Orlicz function given by
$$
N_\theta\left(s\right)=2\int_0^s\int_{-\infty}^\infty\int_{\frac{1}{|a|t}}^\infty |a|f_{\theta}(a)b\frac{e^{-\frac{b^2}{2}}}{\sqrt{2\pi}}\,db\,da\,dt.
$$
Thus
\begin{eqnarray*}
N_\theta\left(\frac1s\right)&=&2\int_0^{\frac1s}\int_{-\infty}^\infty\int_{\frac{1}{|a|t}}^\infty |a|f_{\theta}(a)b\frac{e^{-\frac{b^2}{2}}}{\sqrt{2\pi}}\,db\,da\,dt\cr
&=&\frac{2}{\sqrt{2\pi}}\int_0^{\frac1s}\int_{-\infty}^{\infty} |a|f_\theta(a)e^{-\frac{1}{2a^2t^2}}\,da\,dt\cr
&\geq&\frac{2}{\sqrt{2\pi}}\int_{\frac{1}{2s}}^{\frac1s}\int_{-\infty}^{\infty} |a|f_\theta(a)e^{-\frac{1}{2a^2t^2}}\,da\,dt\cr
&\geq&\frac{2}{\sqrt{2\pi}}\int_{\frac{1}{2s}}^{\frac1s}\int_{-\infty}^{\infty} |a|f_\theta(a)e^{-\frac{2s^2}{a^2}}\,da\,dt\cr
&=&\frac{1}{\sqrt{2\pi}s}\int_{-\infty}^{\infty} |a|f_\theta(a)e^{-\frac{2s^2}{a^2}}\,da\cr
&\geq&\frac{1}{\sqrt{2\pi}s}\int_{\{|a|\geq\sqrt s\}} |a|f_\theta(a)e^{-\frac{2s^2}{a^2}}\,da\cr
&\geq&\frac{1}{\sqrt{2\pi s}}\int_{\{|a|\geq\sqrt s\}}f_\theta(a)e^{-2s}\,da\cr
&=&\frac{e^{-2s}}{\sqrt{2\pi s}} \Pro\left( |\langle X,\theta\rangle|\geq \sqrt{s}\right) .\cr
\end{eqnarray*}
Taking $s=\alpha^2\log N$ we have that for every $\theta\in S^{n-1}$
$$
N_\theta\left(\frac{1}{\alpha^2\log N}\right)\geq\frac{1}{N^{2\alpha^2}\sqrt{2\pi \alpha^2\log N}} \Pro\left(|\langle X,\theta\rangle|\geq \alpha\sqrt{\log N}\right).
$$
By Lemma \ref{ProbEstimateCLT} there exists a constant $\kappa$ such that if $\alpha^2<\frac{\kappa}{\delta}$, then for a set of directions $\Theta$ with $\sigma(\Theta)\geq 1-Ce^{-\sqrt n}$ we have that
$$
\Pro\left(|\langle X,\theta\rangle|\geq \alpha\sqrt{\log N}\right) \geq \frac{c}{N^{\frac{\alpha^2}{2}}\sqrt{\log N}},
$$
and so for this set of directions
$$
N_\theta\left(\frac{1}{\alpha^2\log N}\right)\geq\frac{c}{N^{\frac{5}{2}\alpha^2} \alpha\log N}>\frac{1}{N},
$$
if $\alpha$ is chosen small enough. Consequently, for this set of directions
$$
\E_G\E_{X_1,\dots X_N}h_{K_{N,G}}(\theta)\geq c(\delta)\log N 
$$
and, by Markov's inequality,
$$
\E f(G)=\E_G\E_{X_1,\dots X_N}w(K_{N,G})\geq c(\delta)\log N (1-Ce^{-\sqrt n})\geq c_1(\delta)\log N .
$$
\end{proof}

When the number of vertices is bigger, following the proof in \cite{DGT} where the authors used the ideas in \cite{LPRTJ} to prove a similar result when no perturbations are involved, we have the following stronger result:

\begin{lemma}\label{ExpectationGaussianManyPoints}
Let $X_1,\dots, X_N$ ($N>n^2$) be independent copies of an isotropic log-concave random vector in $\R^n$. There exist an absolute constant $c$ such that if $G=(g_1,\dots,g_N)$ is a standard Gaussian random vector in $\R^N$ then
$$
\Pro_{G,X_1,\dots,X_N}\left(K_{N,G}\supseteq c\sqrt{\log N}Z_{\log N}(X)\right)\to_{n\to\infty}1.
$$
\end{lemma}

\begin{proof}
Let $\Gamma:\ell_2^n\to\ell_2^N$ be the random operator
$$
\Gamma(y)=(g_1\langle X_1,y\rangle,\dots,g_N\langle X_N,y\rangle)
$$
and for every $\gamma>0$, let $\Omega_\gamma$ be the event
$$
\Omega_\gamma=\left\{\Vert\Gamma\Vert\leq\gamma\sqrt N\sqrt{\log N}\right\}.
$$
We have that
\begin{eqnarray*}
\Pro_{G,X_1,\dots,X_N}(\Omega_\gamma^c)&=&\Pro_{G,X_1,\dots,X_N}\left(\max_{\theta\in S^{n-1}}\frac{1}{N}\sum_{i=1}^Ng_i^2\langle X_i,\theta\rangle^2>\gamma^2 \log N\right)\cr
&\leq&\Pro_{G,X_1,\dots,X_N}\left(\max_{1\leq i\leq N}g_i^2\max_{\theta\in S^{n-1}}\frac{1}{N}\sum_{i=1}^N\langle X_i,\theta\rangle^2>\gamma^2 \log N\right)\cr
&\leq&\Pro_G(\max_{1\leq i\leq N}|g_i|>\alpha\sqrt{\log N})+\Pro_{X_1,\dots,X_N}\left(\max_{\theta\in S^{n-1}}\frac{1}{N}\sum_{i=1}^N\langle X_i,\theta\rangle^2>\frac{\gamma^2}{\alpha^2} \right)\cr
&\leq&N\Pro_{g_1}(|g_1|>\alpha\sqrt{\log N})+\Pro_{X_1,\dots,X_N}\left(\max_{\theta\in S^{n-1}}\frac{1}{N}\sum_{i=1}^N\langle X_i,\theta\rangle^2>\frac{\gamma^2}{\alpha^2} \right)\cr
&\leq&\frac{2N}{\sqrt{2\pi}N^{\frac{\alpha^2}{2}}}+\Pro_{X_1,\dots,X_N}\left(\max_{\theta\in S^{n-1}}\frac{1}{N}\sum_{i=1}^N\langle X_i,\theta\rangle^2>\frac{\gamma^2}{\alpha^2} \right).\cr
\end{eqnarray*}

An application of the main Theorem in \cite{MP} gives that if $N\geq c_1n\log^2n$ (which happens for $n$ big enough since we are assuming $N\geq n^2$) then if $\frac{\gamma}{\alpha}$ is a constant big enough
$$
\Pro\left(\max_{\theta\in S^{n-1}}\frac{1}{N}\sum_{i=1}^N\langle X_i,\theta\rangle^2>\frac{\gamma^2}{\alpha^2} \right)\leq e^{-c_2\frac{\gamma}{\alpha}\left(\frac{N}{(\log N)(n\log n)}\right)^\frac14}.
$$
Consequently, if we take $\alpha$ a constant big enough and $\gamma$ a constant big enough we have that
$$
\Pro(\Omega^c)\leq\varepsilon_n
$$
with $\varepsilon_n$ tending to 0 as $n$ goes to $\infty$.

On the other hand, for any $\sigma\subseteq\{1,\dots, N\}$ and any $\theta\in S^{n-1}$, by Paley-Zygmund inequality
\begin{eqnarray*}
&&\Pro_{G,X_1,\dots,X_N}\left(\max_{i\in\sigma}|\langle g_iX_i,\theta\rangle|\leq\frac{1}{2}(\E |g_1|)^\frac{1}{q}(\E|\langle X_1,\theta\rangle|^q)^\frac{1}{q}\right)\cr
&=&\prod_{i\in\sigma}\Pro_{G,X_1,\dots,X_N}\left(|\langle g_iX_i,\theta\rangle|\leq\frac{1}{2}(\E |g_1|)^\frac{1}{q}(\E|\langle X_1,\theta\rangle|^q)^\frac{1}{q}\right)\cr
&\leq&\left(1-\left(1-\left(\frac12\right)^q\right)^2\frac{\left(\E|g_1|^q\E|\langle X_1,\theta\rangle|^q\right)^2}{\E|g_1|^{2q}\E|\langle X_1,\theta\rangle|^{2q}}\right)^{|\sigma|}.
\end{eqnarray*}
Since, by Borell's lemma (see \cite{MS}, Appendix III) there exist absolute constants $C_1,C_2$ such that
$$
\E|g_1|^{2q}\leq C_1^q\E|g_1|^q\hspace{1cm}\E|\langle X_1,\theta\rangle|^{2q}\leq C_2^q\E|\langle X_1,\theta\rangle|^q
$$
the quantity above is bounded by
$$
\left(1-\frac{1}{4C^q}\right)^{|\sigma|}\leq e^{-\frac{|\sigma|}{4C^q}}.
$$

Take $\beta\in(0,\frac{1}{2}]$, $m=\left[8\left(\frac{N}{n}\right)^{2\beta}\right]$ and $k=\left[\frac{N}{m}\right]$. Let $\sigma_1,\dots,\sigma_k$ be a partition of $\{1,\dots, N\}$ with $m\leq|\sigma_i$ for every $i$ and $\Vert\cdot\Vert_0$ be the norm
$$
\Vert u\Vert_0=\frac{1}{k}\sum_{i=1}^k
\max_{j\in\sigma_i}|u_j|.$$
Since for all $1\leq i\leq k$ and every $z\in\R^n$
$$
h_{K_{N,G}}(z)=\max_{1\leq j\leq N}|\langle g_jX_j,z\rangle|\geq\max_{j\in\sigma_i}|\langle g_jX_j,z\rangle|,
$$
then for every $z\in\R^n$
$$
h_{K_{N,G}}(z)\geq\Vert\Gamma(z)\Vert_0.
$$

By Markov's inequality, if $z\in\R^n$ verifies $\Vert\Gamma(z)\Vert_0\leq\frac{1}{4}\left(\E|g_1|^q\right)^\frac{1}{q}\left(\E|\langle X_1,z\rangle|^q\right)^\frac{1}{q}$ then there exists a set $I\subseteq\{1,\dots,k\}$ with $|I|\geq\frac{k}{2}$ such that
$$
\max_{j\in\sigma_i}|\langle g_jX_j,z\rangle|\leq\frac{1}{2}\left(\E|g_1|^q\right)^\frac{1}{q}\left(\E|\langle X_1,z\rangle|^q\right)^\frac{1}{q}
$$
for every $i\in I$. Thus, for every $z\in\R^n$
\begin{eqnarray*}
&&\Pro_{G,X_1,\dots,X_N}\left(\Vert\Gamma(z)\Vert_0\leq\frac{1}{4}\left(\E|g_1|^q\right)^\frac{1}{q}\left(\E|\langle X_1,z\rangle|^q\right)^\frac{1}{q}\right)\cr
&\leq&\sum_{|I|=\lceil\frac{k}{2}\rceil}\Pro_{G,X_1,\dots,X_N}\left(\max_{j\in\sigma_i}|\langle g_jX_j,z\rangle|\leq\frac{1}{2}\left(\E|g_1|^q\right)^\frac{1}{q}\left(\E|\langle X_1,z\rangle|^q\right)^\frac{1}{q}\,\forall i\in I\right)\cr
&\leq&\sum_{|I|=\lceil\frac{k}{2}\rceil}\prod_{i\in I}\Pro_{G,X_1,\dots,X_N}\left(\max_{j\in\sigma_i}|\langle g_jX_j,z\rangle|\leq\frac{1}{2}\left(\E|g_1|^q\right)^\frac{1}{q}\left(\E|\langle X_1,z\rangle|^q\right)^\frac{1}{q}\right)\cr
&\leq&\sum_{|I|=\lceil\frac{k}{2}\rceil}\prod_{i\in I}e^{-\frac{|\sigma_i|}{4C^q}}\leq\left(\begin{array}{c}k\cr\lceil\frac{k}{2}\rceil\end{array}\right)e^{-\frac{ckm}{C^q}}\leq e^{c'k-\frac{ckm}{C^q}}\leq e^{-cN^{1-\beta}n^\beta}.
\end{eqnarray*}
if we take $q\sim\beta\log\frac{N}{n}$.

Now, let $S=\{z\in\R^n\,:\,\frac{1}{2}\left(\E|g_1|^q\right)^\frac{1}{q}\left(\E|\langle X_1,z\rangle|^q\right)^\frac{1}{q}=1\}$ and let $U$ be a $\delta$-net with cardinality $|U|\leq\left(\frac{3}{\delta}\right)^n$.i.e., for every $z\in S$ there is $u\in U$ such that $\frac{1}{2}\left(\E|g_1|^q\right)^\frac{1}{q}\left(\E|\langle X_1,z-u\rangle|^q\right)^\frac{1}{q}\leq \delta$. Then
$$
\Pro\left(\exists u\in U\,:\,\Vert\Gamma(u)\Vert_0\leq\frac{1}{2}\right)\leq e^{n\log\frac{3}{\delta}-cN^{1-\beta}n^\beta}.
$$
Now, if $\Gamma\in\Omega_\gamma$, since $\left(\E|\langle X_1,z\rangle|^q\right)^\frac{1}{q}\geq|z|$ for all $z\in\R^n$, we have
\begin{eqnarray*}
\Vert\Gamma(z)\Vert_0\leq\frac{1}{\sqrt k}|\Gamma(z)|&\leq&\gamma\sqrt{\frac Nk}\sqrt{\log N}|z|\leq\gamma\sqrt{\frac Nk}\sqrt{\log N}\left(\E|\langle X_1,z\rangle|^q\right)^\frac{1}{q}\cr
&\leq& \frac{C}{\sqrt\beta}\gamma\sqrt{\frac Nk}(\E|g_1|^q)^\frac{1}{q}\left(\E|\langle X_1,z\rangle|^q\right)^\frac{1}{q},
\end{eqnarray*}
since $\sqrt{\log N}\sim\sqrt{\log\frac{N}{n}}\sim\frac{1}{\sqrt\beta}(\E|g_1|^q)^\frac{1}{q}$ because $N\geq n^2$. Thus, for every $z\in S$ there exists $u\in U$ such that $\frac{1}{2}\left(\E|g_1|^q\right)^\frac{1}{q}\left(\E|\langle X_1,z-u\rangle|^q\right)^\frac{1}{q}\leq \delta$ and so
$$
\Vert\Gamma(u)\Vert_0\leq\Vert\Gamma(z)\Vert_0+C\frac{\gamma}{\sqrt\beta}\sqrt{\frac Nk}\delta.
$$
Choosing $\delta=\frac{\sqrt {\beta k}}{4C\gamma\sqrt N}$ we have that
\begin{eqnarray*}
&&\Pro_{G,X_1,\dots,X_N}\left(\Gamma\in\Omega_\gamma\,:\,\exists z\in\R^n\,:\,\Vert\Gamma(z)\Vert_0\leq\frac{1}{8}\left(\E|g_1|^q\right)^\frac1q\left(\E|\langle X_1,z\rangle|^q\right)^\frac1q\right)\cr
&=&\Pro_{G,X_1,\dots,X_N}\left(\Gamma\in\Omega_\gamma\,:\,\exists z\in S\,:\,\Vert\Gamma(z)\Vert_0\leq\frac{1}{4}\right)\cr
&\leq&\Pro_{G,X_1,\dots,X_N}\left(\Gamma\in\Omega_\gamma\,:\,\exists u\in U\,:\,\Vert\Gamma(z)\Vert_0\leq\frac{1}{2}\right)\cr
&\leq&e^{n\log\frac{12C\gamma\sqrt N}{\sqrt {\beta k}}-cN^{1-\beta}n^\beta}\leq e^{-cN^{1-\beta}n^\beta}
\end{eqnarray*}
if $N\geq C(\beta) n$.

Consequently, choosing $\beta$ a constant  in $(0,\frac{1}{2}]$ with probability greater than $1-e^{-cN^{1-\beta}n^\beta}-\varepsilon_n$ we have that
$$
K_{N,G}\supseteq \frac{1}{8}(\E|g_1|^q)^\frac{1}{q}Z_q(X)\supseteq c\sqrt{\log N}Z_{\log N}(X)
$$
\end{proof}

\begin{cor}\label{CorExpectationGaussian}
Let $X_1,\dots, X_N$ ($n\leq N\leq e^{\sqrt n}$) be independent copies of an isotropic log-concave random vector in $\R^n$, and let $f:\R^N\to\R$ be the function $f(y)=\E_{X_1,\dots,X_N}w(K_{N,y})$. There exist absolute constants $c_1,c_2$ such that if $G=(g_1,\dots,g_N)$ is a standard Gaussian random vector in $\R^N$ then
$$
c_1\log N\leq\E f(G)\leq c_2\log N.
$$
\end{cor}
\begin{proof}
If $n\leq N\leq n^2$ this is Lemma \ref{LemmaExpectationGaussianFewVertices}. If $n^2\leq N\leq e^{\sqrt n}$ by Lemma \ref{ExpectationGaussianManyPoints}
$$
\Pro_{G,X_1,\dots, X_N}(w(K_{N,G})\geq c\sqrt{\log N}w(Z_{\log N}(X)))
$$
tends to 1 as $n\to\infty$. Thus, by Markov's inequality
$$
\E f(G)\geq c'\sqrt{\log N}w(Z_{\log N}(X))\sim\log N.
$$
\end{proof}
By integration in polar coordinates, we obtain the following:

\begin{cor}\label{LemmaExpectationSphere}
Let $X_1,\dots, X_N$ ($n\leq N\leq e^{\sqrt n}$) be independent copies of an isotropic log-concave random vector in $\R^n$, and let $f:\R^N\to\R$ be the function $f(y)=\E_{X_1,\dots,X_N}w(K_{N,y})$. There exists  absolute constants $c_1,c_2$ such that, if $u$ is random vector uniformly distributed on $S^{N-1}$, then
$$
c_1\frac{\log N}{\sqrt N}\leq\int_{S^{N-1}} f(u)d\sigma(u)\leq c_2\frac{\log N}{\sqrt N}.
$$
\end{cor}

In order to compute $\E f(y)$ when $y$ is uniformly distributed in $B_p^N$ we need the following lemma, which was proved in \cite{SZ}:

\begin{lemma} [\cite{SZ}, Lemma 2] \label{LEM_expectation_max_schechtman_zinn}
Let $g_1,\ldots,g_N$ be independent identically distributed random variables with density $f_{g}(t) = \frac{1}{2\Gamma(1+\frac{1}{p})}e^{-|t|^p}$. For all $1\leq p \leq q<\infty$, if $N\geq 20p\Gamma\left(1+\frac{1}{p}\right)$, then $\mathbb E \left( \sum_{i=1}^N |g_i|^q\right)^{\frac{1}{q}}$ is equivalent, up to absolute constants, to $q^{\frac{1}{p}}N^{\frac{1}{q}}$, if $q\leq \log N$, and to $(\log N)^{\frac{1}{p}}$, otherwise.
\end{lemma}

Consequently, if $p\leq CN$, we have that
$$
\E\max_{1\leq i\leq N}|g_i|\simeq(\log N)^\frac1p.
$$

\begin{lemma}\label{LemmaExpectationB_p}
Let $X_1,\dots, X_N$ ($n\leq N\leq e^{\sqrt n}$) be independent copies of an isotropic log-concave random vector in $\R^n$, and let $f:\R^N\to\R$ be the function $f(y)=\E_{X_1,\dots,X_N}w(K_{N,y})$. There exist absolute constants $c_1$, $c_2$ such that for any $1\leq p\leq\infty$, if $y$ is a random vector uniformly distributed on $B_p^N$,
$$
c_1\frac{(\log N)^{\frac{1}{p}+\frac{1}{2}}}{N^\frac{1}{p}}\leq\E_y f(y)\leq c_2\frac{(\log N)^{\frac{1}{p}+\frac{1}{2}}}{N^\frac{1}{p}}.
$$
\end{lemma}

\begin{proof}
Let $G=(g_1,\dots,g_N)$ be a random vector where $g_1,\dots g_N$ are independent identically distributed random variables with density with respect to the Lebesgue measure given by
$$
f_g(t)=\frac{e^{-|t|^p}}{2\Gamma\left(1+\frac{1}{p}\right)}.
$$
For any $\theta\in S^{n-1}$ and any $1\leq p<\infty$, by Fubini's theorem,
\begin{eqnarray*}
\E_G\E_{X_1,\dots,X_N} h_{K_{N,G}}(\theta)&=&\int_{\R^N}\E_{X_1,\dots,X_N}h_{K_{N,x}}(\theta)\frac{e^{-\Vert x\Vert_p^p}}{\left(2\Gamma\left(1+\frac{1}{p}\right)\right)^N}\,dx\cr
&=&\int_{\R^N}\E_{X_1,\dots,X_N}h_{K_{N,x}}(\theta)\int_{\Vert x\Vert_p^p}^\infty\frac{e^{-t}}{\left(2\Gamma\left(1+\frac{1}{p}\right)\right)^N}\,dt\,dx\cr
&=&\int_0^\infty\frac{e^{-t}}{\left(2\Gamma\left(1+\frac{1}{p}\right)\right)^N}\int_{t^{\frac{1}{p}}B_p^N}\E_{X_1,\dots,X_N}h_{K_{N,x}}(\theta)\,dx\,dt\cr
&=&\int_0^\infty\frac{t^{\frac{N+1}{p}}e^{-t}}{\left(2\Gamma\left(1+\frac{1}{p}\right)\right)^N}\int_{B_p^N}\E_{X_1,\dots,X_N}h_{K_{N,y}}(\theta)\,dy\,dt\cr
&=&\frac{\Gamma\left(1+\frac{N+1}{p}\right)}{\left(2\Gamma\left(1+\frac{1}{p}\right)\right)^N}\int_{B_p^N}\E_{X_1,\dots,X_N}h_{K_{N,y}}(\theta)\,dy\cr
&=&\frac{\Gamma\left(1+\frac{N+1}{p}\right)}{\Gamma\left(1+\frac{N}{p}\right)}\E_y\E_{X_1,\dots,X_N}h_{K_{N,y}}(\theta)\cr
& \sim & N^\frac{1}{p}\E_y\E_{X_1,\dots,X_N}h_{K_{N,y}}(\theta).
\end{eqnarray*}
On one hand, if $1\leq p\leq\log N$, by Lemma \ref{LEM_expectation_max_schechtman_zinn}
\begin{eqnarray*}
\E_G\E_{X_1,\dots,X_N} h_{K_{N,G}}(\theta)&=&\E_G\E_{X_1,\dots,X_N}\max_{1\leq i\leq N}|\langle g_iX_i,\theta\rangle|\cr&\leq&\E_G\max_{1\leq i\leq N}|g_i|\E_{X_1,\dots,X_N}\max_{1\leq i\leq N}|\langle X_i,\theta\rangle|\cr &\sim& (\log N)^\frac{1}{p}\E_{X_1,\dots,X_N} h_{K_N}(\theta),
\end{eqnarray*}
and integrating in $\theta\in S^{n-1}$ we obtain that there exists an absolute constant such that
$$
\E_y f(y)\leq c_2\frac{(\log N)^{\frac{1}{p}+\frac{1}{2}}}{N^\frac{1}{p}}.
$$
On the other hand, like before, we use different techniques to prove the lower estimate depending on the number of vertices. Assume first that $n\leq N\leq n^\delta$.
$$
\E_G\E_{X_1,\dots,X_N} h_{K_{N,G}}(\theta)=\E_G\E_{X_1,\dots,X_N}\max_{1\leq i\leq N}|\langle g_iX_i,\theta\rangle| \sim \Vert(1,\dots,1)\Vert_{N_\theta},
$$
where
$$
N_\theta\left(s\right)=2\int_0^{s}\int_{-\infty}^{\infty}\int_{\frac{1}{|a|t}}^\infty |a|bf_\theta(a)\frac{e^{-b^p}}{2\Gamma\left(1+\frac{1}{p}\right)}\,db\,da\,dt.
$$
Notice that
\begin{eqnarray*}
N_\theta\left(\frac{1}{s}\right)&=&2\int_0^{\frac{1}{s}}\int_{-\infty}^{\infty}\int_{\frac{1}{|a|t}}^\infty |a|bf_\theta(a)\frac{e^{-b^p}}{2\Gamma\left(1+\frac{1}{p}\right)}\,db\,da\,dt\cr
&\geq&\frac{1}{\Gamma\left(1+\frac{1}{p}\right)}\int_{\frac{1}{2^\frac1ps}}^{\frac{1}{s}}\int_{-\infty}^{\infty}\int_{\frac{1}{at}}^\infty |a|bf_\theta(a)e^{-b^p}\,db\,da\,dt\cr
&\geq&\frac{(2^\frac{1}{p}-1)}{2^\frac1ps\Gamma\left(1+\frac{1}{p}\right)}\int_{-\infty}^{\infty}\int_{\frac{2^\frac1ps}{a}}^\infty |a|bf_\theta(a)e^{-b^p}\,db\,da\cr
&\geq&\frac{(2^\frac{1}{p}-1)}{2^\frac1ps\Gamma\left(1+\frac{1}{p}\right)}\int_{\{|a|\geq s^{\frac{p}{p+2}}\}}\int_{\frac{2^\frac1ps}{a}}^\infty |a|bf_\theta(a)e^{-b^p}\,db\,da\cr
&\geq&\frac{(2^\frac{1}{p}-1)}{2^\frac1ps^{\frac{2}{p+2}}\Gamma\left(1+\frac{1}{p}\right)}\Pro\left(|\langle X,\theta\rangle|\geq s^{\frac{p}{p+2}}\right)\int_{2^\frac1ps^{\frac{2}{p+2}}}^\infty be^{-b^p}\,db.\cr
\end{eqnarray*}
Taking $s_0=(\alpha^2\log N)^{\frac{1}{p}+\frac12}$, we have
$$
N_\theta\left(\frac{1}{s_0}\right)\geq\frac{(2^\frac{1}{p}-1)}{2^\frac1p(\alpha^2\log N)^\frac1p\Gamma\left(1+\frac{1}{p}\right)}\Pro\left(|\langle X,\theta\rangle| \geq \alpha\sqrt{\log N}\right)\int_{(2\alpha^2\log N)^\frac{1}{p}}^\infty be^{-b^p}\,db.
$$
By Lemma \ref{ProbEstimateCLT} there exists a constant $\kappa$ such that if $\alpha^2<\frac{\kappa}{\delta}$, then for a set of directions $\Theta$ with $\sigma(\Theta)\geq 1-Ce^{-\sqrt n}$ we have that
$$
\Pro\left( |\langle X,\theta\rangle|\geq \alpha\sqrt{\log N} \right) \geq\frac{c}{N^{\frac{\alpha^2}{2}}\sqrt{\log N}}.
$$
Besides, if $1\leq p\leq 2$,
\begin{eqnarray*}
\int_{(2\alpha^2\log N)^\frac{1}{p}}^\infty be^{-b^p}\,db&=&\int_{(2\alpha^2\log N)^\frac{1}{p}}^\infty b^{2-p}b^{p-1}e^{-b^p}\,db\geq \frac{(2\alpha^2\log N)^\frac{2-p}{p}}{pN^{2\alpha^2}},
\end{eqnarray*}
and so, choosing $\alpha$ a constant small enough with $\alpha^2\leq\frac{\kappa}{\delta}$,
$$
N_\theta\left(\frac{1}{s_0}\right)>\frac{1}{N},
$$
and for every $\theta\in \Theta$
$$
\E_G\E_{X_1,\dots,X_N} h_{K_{N,G}}(\theta)\geq c_1(\delta)(\log N)^{\frac{1}{p}+\frac12}.
$$
If $2\leq p\leq\log N$
\begin{eqnarray*}
\int_{(2\alpha^2\log N)^\frac{1}{p}}^\infty be^{-b^p}db&=&\int_{(2\alpha^2\log N)^\frac{1}{p}}^\infty b^{2-p}b^{p-1}e^{-b^p}\,db\cr
&=&\frac{1}{(2\alpha^2\log N)^{\frac{p-2}{p}}pN^{2\alpha}}-\frac{p-2}{p}\int_{(2\alpha^2\log N)^\frac{1}{p}}^\infty \frac{b^{p-1}}{b^{2p-2}}e^{-b^p}\,db\cr
&=&\frac{1}{(2\alpha^2\log N)^{\frac{p-2}{p}}pN^{2\alpha^2}}-\frac{p-2}{p^2(2\alpha^2\log N)^{\frac{2p-2}{p}}N^{2\alpha^2}}\cr
&=&\frac{1}{(2\alpha^2\log N)^{\frac{p-2}{p}}pN^{2\alpha^2}}\left(1-\frac{p-2}{p(2\alpha^2\log N)}\right)\cr
\end{eqnarray*}
and also in this case, choosing $\alpha$ a constant small enough with $\alpha^2\leq\frac{\kappa}{\delta}$,
$$
N_\theta\left(\frac{1}{s_0}\right)>\frac{1}{N},
$$
and for every $\theta\in \Theta$
$$
\E_G\E_{X_1,\dots,X_N} h_{K_{N,G}}(\theta)\geq c_1(\delta)(\log N)^{\frac{1}{p}+\frac12}.
$$
Integrating on $S^{n-1}$, by Markov's inequality,
$$
\E_G f(G)\geq c_1(\delta)(\log N)^{\frac{1}{p}+\frac12},
$$
and so
$$
\E_y f(y)\geq c_1(\delta)\frac{(\log N)^{\frac{1}{p}+\frac{1}{2}}}{N^\frac{1}{p}}.
$$
If $p\geq \log N$ the estimate we have to prove is
$$
c_1(\delta)\sqrt{\log N}\leq\E_{B_p^N} f(y)\leq c_2\sqrt{\log N}.
$$
Since for every $y\in B_p^N$ $\max_{1\leq i\leq N}|y_i|\leq 1$,
$$f(y)\leq f(1,\dots,1)=\E_{X_1,\dots, X_N}w(K_N)\sim\sqrt{\log N}.$$
On the other hand
\begin{eqnarray*}
N_\theta\left(\frac{1}{s}\right)&=&2\int_0^{\frac{1}{s}}\int_{-\infty}^{\infty}\int_{\frac{1}{|a|t}}^\infty |a|bf_\theta(a)\frac{e^{-b^p}}{2\Gamma\left(1+\frac{1}{p}\right)}\,db\,da\,dt\cr
&\geq&\frac{1}{\Gamma\left(1+\frac{1}{p}\right)}\int_{\frac{1}{2s}}^{\frac{1}{s}}\int_{-\infty}^{\infty}\int_{\frac{1}{at}}^\infty |a|bf_\theta(a)e^{-b^p}\,db\,da\,dt\cr
&\geq&\frac{1}{2s\Gamma\left(1+\frac{1}{p}\right)}\int_{-\infty}^{\infty}\int_{\frac{2s}{a}}^\infty |a|bf_\theta(a)e^{-b^p}\,db\,da\cr
&\geq&\frac{1}{2s\Gamma\left(1+\frac{1}{p}\right)}\int_{\{|a|\geq s^{\frac{p}{p+2}}\}}\int_{\frac{2s}{a}}^\infty |a|bf_\theta(a)e^{-b^p}\,db\,da\cr
&\geq&\frac{1}{2s^{\frac{2}{p+2}}\Gamma\left(1+\frac{1}{p}\right)}\Pro\left(|\langle X,\theta\rangle|\geq s^{\frac{p}{p+2}}\right)\int_{2s^{\frac{2}{p+2}}}^\infty be^{-b^p}\,db .\cr
\end{eqnarray*}
Taking $s_0=(\alpha^2\log N)^{\frac{1}{p}+\frac12}$, we have
$$
N_\theta\left(\frac{1}{s_0}\right)\geq\frac{1}{2(\alpha^2\log N)^\frac1p\Gamma\left(1+\frac{1}{p}\right)}\Pro\left(|\langle X,\theta\rangle|\geq \alpha\sqrt{\log N}\right)\int_{2(\alpha^2\log N)^\frac{1}{p}}^\infty be^{-b^p}\,db.
$$
By Lemma \ref{ProbEstimateCLT} there exists a constant $\kappa$ such that if $\alpha^2<\frac{\kappa}{\delta}$, then for a set of directions $\Theta$ with $\sigma(\Theta)\geq 1-Ce^{-\sqrt n}$ we have that
$$
\Pro\left( |\langle X,\theta\rangle|\geq \alpha\sqrt{\log N}\right) \geq\frac{c}{N^{\frac{\alpha^2}{2}}\sqrt{\log N}}.
$$
Since $p\geq\log N$, $2(\alpha^2\log N)^\frac{1}{p}$ is smaller than some constant $C$ and so
$$
\int_{2(\alpha^2\log N)^\frac{1}{p}}^\infty be^{-b^p}\,db\geq\int_{C}^{2C}be^{-b^p}\,db\geq c^\prime.
$$
and like in the other cases, taking $\alpha$ a constant small enough and integrating on $S^{n-1}$ we obtain the result.

Now assume that $n^2\leq N\leq e^{\sqrt n}$. The proof in this case follows the one of lemma \ref{ExpectationGaussianManyPoints} so we just sketch it. Let $\Gamma:\ell_2^n\to\ell_2^N$ be the random operator
$$
\Gamma(y)=(g_1\langle X_1,y\rangle,\dots,g_N\langle X_N,y\rangle)
$$
and for every $\gamma>0$, let $\Omega_\gamma$ be the event
$$
\Omega_\gamma=\{\Gamma\,:\,\Vert\Gamma\Vert\leq\gamma\sqrt N(\log N)^\frac{1}{p}\}.
$$
We have that if $N\geq c_1n\log^2n$
\begin{eqnarray*}
\Pro_{G,X_1,\dots,X_N}(\Omega_\gamma^c)&\leq&\Pro_G\left(\max_{1\leq i\leq N}|g_i|>\alpha(\log N)^\frac{1}{p}\right)+\Pro_{X_1,\dots,X_N}\left(\max_{\theta\in S^{n-1}}\frac{1}{N}\sum_{i=1}^N\langle X_i,\theta\rangle^2>\frac{\gamma^2}{\alpha^2} \right)\cr
&\leq&N\Pro\left(|g_1|>\alpha(\log N)^\frac{1}{p}\right)+\Pro_{X_1,\dots,X_N}\left(\max_{\theta\in S^{n-1}}\frac{1}{N}\sum_{i=1}^N\langle X_i,\theta\rangle^2>\frac{\gamma^2}{\alpha^2} \right)\cr
&\leq&\frac{1}{p\Gamma\left(1+\frac{1}{p}\right)N^{\alpha^p-1}}+e^{-c_2\frac{\gamma}{\alpha}\left(\frac{N}{(\log N)(n\log n)}\right)^\frac14}.\cr
\end{eqnarray*}

Consequently, if we take $\alpha$ a constant big enough and $\gamma$ a constant big enough we have that
$$
\Pro(\Omega^c)\leq\varepsilon_n
$$
with $\varepsilon_n$ tending to 0 as $n$ goes to $\infty$.
Like in the proof of Lemma \ref{ExpectationGaussianManyPoints}, if $\beta\in (0,\frac{1}{2}]$, $q\sim\beta\log\frac{N}{n}$ and $\Gamma\in\Omega_\gamma$ we have
\begin{eqnarray*}
\Vert\Gamma(z)\Vert_0\leq\frac{1}{\sqrt k}|\Gamma(z)|&\leq&\gamma\sqrt{\frac Nk}(\log N)^\frac{1}{p}|z|\leq\gamma\sqrt{\frac Nk}\sqrt{\log N}\left(\E|\langle X_1,z\rangle|^q\right)^\frac{1}{q}\cr
&\leq& \frac{C}{\beta^\frac1p}\gamma\sqrt{\frac Nk}(\E|g_1|^q)^\frac{1}{q}\left(\E|\langle X_1,z\rangle|^q\right)^\frac{1}{q},
\end{eqnarray*}
since $(\log N)^\frac{1}{p}\sim\left(\log\frac{N}{n}\right)^\frac{1}{p}\sim\frac{1}{\beta^\frac{1}{p}}(\E|g_1|^q)^\frac{1}{q}$ because $N\geq n^2$. Thus, for every $z\in S$ there exists $u\in U$ such that $\frac{1}{2}\left(\E|g_1|^q\right)^\frac{1}{q}\left(\E|\langle X_1,z-u\rangle|^q\right)^\frac{1}{q}\leq \delta$ and so
$$
\Vert\Gamma(u)\Vert_0\leq\Vert\Gamma(z)\Vert_0+C\frac{\gamma}{\beta^{\frac1p}}\sqrt{\frac Nk}\delta.
$$
Choosing $\delta=\frac{\beta^{\frac1p}\sqrt k}{4C\gamma\sqrt N}$ we have that
$$
\Pro_{G,X_1,\dots,X_N}\left(\Gamma\in\Omega_\gamma\,:\,\exists z\in\R^n\,:\,\Vert\Gamma(z)\Vert_0\leq\frac{1}{8}\left(\E|g_1|^q\right)^\frac1q\left(\E|\langle X_1,z\rangle|^q\right)^\frac1q\right)
\leq e^{-cN^{1-\beta}n^\beta}
$$
if $N\geq C(\beta) n$.

Consequently, choosing $\beta$ a constant  in $(0,\frac{1}{2}]$ with probability greater than $1-e^{-cN^{1-\beta}n^\beta}-\varepsilon_n$ we have that
$$
K_{N,G}\supseteq \frac{1}{8}(\E|g_1|^q)^\frac{1}{q}Z_q(X)\supseteq c(\log N)^{\frac1p}Z_{\log N}(X)
$$

Using Markov's inequality we obtain the desired estimate.

In the case $p=\infty$ we proceed in the same way taking $G$ a random vector uniformly distributed in $B_\infty^N$.
\end{proof}

Let us now compute the expected value of $f$ when $y=(\xi_1,\dots,\xi_N)$, where $\xi_1,\ldots,\xi_N$ are independent copies of a $p$-stable random variable. It was proved in \cite{GLSW2} that for $1<p<2$ a sequence of independent copies of a $p$-stable random variables $\xi$ generates the corresponding $\ell_p$-norm, {\it i.e.}, if $\xi_1,\ldots,\xi_N$ are independent copies of a $p$-stable random variable, then
\begin{equation}\label{EQU_p_norm}
\mathbb E \max_{1\leq i \leq N} |x_i\xi_i| \sim \|x\|_p,
\end{equation}
for any $x\in\R^N$. The corresponding result for $\ell_p$-norms where $p\geq 2$ was proved in \cite{PR}, using log-$\gamma_{1,p}$ distributed random variables.

\begin{lemma}\label{Lemma_Expectation_p-stable}
Let $X_1,\dots, X_N$ be independent copies of an isotropic log-concave random vector in $\R^n$, and let $f:\R^N\to\R$ be the function $f(y)=\E_{X_1,\dots,X_N}w(K_{N,y})$. Let $1<p<2$ and $y=(\xi_1,\ldots,\xi_N)$, where $\xi_1,\dots,\xi_N$ are independent copies of a $p$-stable random variable $\xi$. Then
$$
\mathbb E_y f(y) \sim N^{\frac{1}{p}}.
$$
\end{lemma}

\begin{proof}
First of all, let $\theta \in S^{n-1}$. Then, using (\ref{EQU_p_norm}), we have
$$
\mathbb E_y \mathbb E_{X_1,\dots,X_N} h_{K_{N,y}}(\theta) = \mathbb E_y \mathbb E_{X_1,\dots,X_N} \max_{1\leq i \leq N} |\xi_i \langle X_i,\theta \rangle| \sim \mathbb E_{X_1,\dots,X_N} \| (\langle X_i,\theta\rangle)_{i=1}^N\|_p.
$$
Therefore, using Jensen's inequality, we obtain that
$$
\mathbb E_y \mathbb E_{X_1,\dots,X_N} h_{K_{N,y}}(\theta) \geq c  \| (\mathbb E_{X_i} |\langle X_i,\theta\rangle|)_{i=1}^N\|_p.
$$
Since, as a consequence of Borell's lemma, $\mathbb E_X |\langle X_i,\theta\rangle| \sim1$, we obtain the lower bound
$$
\mathbb E_y \mathbb E_{X_1,\dots, X_N} h_{K_{N,y}}(\theta) \geq c N^{\frac{1}{p}}.
$$
On the other hand, using H\"older's inequality, we get
$$
\mathbb E_{X_1,\dots,X_N} \| (\langle X_i,\theta\rangle)_{i=1}^N\|_p \leq  \left( \sum_{i=1}^N \mathbb E_X |\langle X_i,\theta\rangle|^p\right)^{\frac{1}{p}}.
$$
Again, since by H\"older's inequality $\mathbb E_{X_1,\dots,X_N} |\langle X_i,\theta\rangle|^p \leq 1$, we obtain the upper bound
$$
\mathbb E_y \mathbb E_X h_{K_{N,y}}(\theta) \leq C N^{\frac{1}{p}}.
$$
The bounds do not depend on the direction $\theta\in S^{n-1}$. Therefore, taking the average on the sphere and using Fubini's theorem we get
$$
\mathbb E_y f(y) = \int_{S^{n-1}} \mathbb E_y \mathbb E_{X_1,\dots, X_N} h_{K_{N,y}}(\theta) \, d\sigma(\theta) \sim N^{\frac{1}{p}}.
$$
\end{proof}

\begin{rmk}
Exchanging $p$-stable by log-$\gamma_{1,p}$ distributed random variables and using the results from \cite{PR}, Lemma \ref{Lemma_Expectation_p-stable} can be obtained for $p\geq 2$ with constants only depending on $p$.
\end{rmk}

As mentioned before, there is an Orlicz norm $\Vert\cdot\Vert_{M_\theta}$ associated to every direction $\theta$ in the unit sphere. The proofs of the three previous lemmas give us the following properties of these norms:

\begin{cor}
Let $X$ be an isotropic log-concave random vector in $\R^n$ and for every $\theta$ let $\Vert\cdot\Vert_{M_\theta}$ be the Orlicz norm in $\R^N$ defined in Theorem \ref{THM Schuett Werner Litvak Gordon}. Then
\begin{itemize}
\item[a)] For every $\theta$ and every $N$, if $y=(\xi_1,\dots,\xi_N)$ is a random vector where $\xi_1,\dots,\xi_N$ are independent copies of a $p$-stable random variable $\xi$ ($1<p<2$),
    $$
    \E_y\Vert y\Vert_{M_\theta}\sim N^\frac{1}{p}
    $$
\item[b)] There exists a set $\Theta\subset S^{n-1}$ with $\sigma(\Theta)\geq 1-Ce^{-\sqrt n}$ such that if $n\leq N\leq n^\delta$ and $G$ is a Gaussian random vector in $\R^N$ or $y$ is a random vector uniformly distributed in $B_p^N$, then, for every $\theta\in\Theta$,
    $$
    c(\delta)\log N \leq \E_G\Vert G\Vert_{M_\theta}\leq C\log N
    $$
and
    $$
    c(\delta)\frac{(\log N)^{\frac{1}{p}+\frac{1}{2}}}{N^\frac{1}{p}}\leq\E_y\Vert y\Vert_{M_\theta}\leq C\frac{(\log N)^{\frac{1}{p}+\frac{1}{2}}}{N^\frac{1}{p}}.
    $$
\end{itemize}
\end{cor}
\begin{proof}
The proof is contained in the three previous lemmas, using that $\Vert y\Vert_{M_\theta}\sim\E_{X_1,\dots,X_N} h_{K_{N,y}}(\theta)$.
The only thing left to prove is that if $\theta\in\Theta$ then $\E_{X_1,\dots,X_N}h_{K_N}(\theta)\leq C\sqrt{\log N}$, which is a consequence of the central limit theorem. Let us denote by $K_\delta$ the floating body defined by
$$
h_{K_\delta}(\theta)=\sup\{t>0\,:\,\Pro\left( |\langle x,\theta\rangle|\geq t\right)\geq \delta\}.
$$
It was proved in \cite{PW} that $K_\delta$ is homothetic to $Z_{\log\frac1\delta}$ with absolute constants. Then
$$
\E_{X_1,\dots,X_N}h_{K_N}(\theta)\leq C\E\left(\sum_{i=1}^N|\langle X_i,\theta\rangle|^{\log N}\right)^{\frac{1}{\log N}} \leq C h_{Z_{\log N}(X)}(\theta)\sim h_{K_{\frac1N}}(\theta)
$$
and if $\theta$ is in the set $\Theta$ given by the central limit theorem, then
\begin{eqnarray*}
\Pro\left(|\langle X,\theta\rangle|\geq\beta\sqrt{\log N}\right) &=& \int_{-\beta\sqrt{\log N}}^{\beta\sqrt{\log N}}f_\theta(t)dt\cr
&\leq&\left(1+\frac{1}{n^\kappa}\right)2\int_0^{\beta\sqrt{\log N}}\gamma(t)dt\cr
&\leq&\frac{8}{N^\frac{\beta^2}{2}\beta\sqrt{\log N}}<\frac{1}{N},
\end{eqnarray*}
if $\beta$ is a constant big enough.
\end{proof}

Now we can apply the concentration of measure results and prove the theorems.

\begin{proof}[Proof of Theorems \ref{THM_Gaussian_perturbation}, \ref{THM_Sphere_perturbation}, \ref{THM_B_p_perturbation} and \ref{THM_p-stable_perturbation}:]
Let $f:\R^N\to \R$ be
$$
f(y)=\E_{X_1,\dots,X_N}w(K_{N,y}).
$$
By Theorem \ref{THM_Gaussian_concentration}, for any $t>0$
$$
\Pro_G\left(\left|\frac{f(G)}{\E_G f(G)}-1\right|\leq t\right)\geq 1-e^{-\frac{ct^2(\E_Gf(G))^2}{L^2}},
$$
where $L$ is the Lipschitz constant of $f$. Thus
$$
\Pro_G\left((1-t)\E_{G}f(G)\leq f(G)\leq (1+t)\E_Gf(G)\right)\geq 1-e^{-\frac{ct^2(\E_Gf(G))^2}{L^2}},
$$
Applying Lemmas \ref{LemmaLipsichtzConstant} and Corollary \ref{CorExpectationGaussian}to estimate $L$ and $\E_Gf(G)$ we have
$$
\Pro_G\left(c_1(\delta)(1-t)\log N\leq f(G)\leq c_2(1+t)\log N\right)\geq 1-e^{-c(\delta)t^2\log N}.
$$
In the same way, applying Theorem \ref{THM_Sphere_concentration} and Lemmas \ref{LemmaLipsichtzConstant} and Corollary \ref{LemmaExpectationSphere} we prove Theorem \ref{THM_Sphere_perturbation}.
Applying Theorem \ref{THM_B_p_concentration} and Lemmas \ref{LemmaLipsichtzConstant} and \ref{LemmaExpectationB_p} we prove Theorem \ref{THM_B_p_perturbation}.
Finally, applying Theorem \ref{THM_p_stable_concentration} and Lemmas \ref{LemmaLipsichtzConstant} and \ref{Lemma_Expectation_p-stable} we obtain the proof of Theorem \ref{THM_B_p_perturbation}.
\end{proof}

\section{Partial Results for Arbitrary Perturbations of Random Polytopes}\label{ArbitraryPerturbations}

In this section we will give a lower estimate for the expected value of the mean width of $K_{N,y}$ when $y$ is not a random vector. We will prove the following:

\begin{proposition}
Let $X_1,\dots, X_N$ be independent copies of an isotropic log-concave random vector in $\R^n$. There exist absolute constants $c_1,c_2$ such that for any $y\in \R^N$, if we call
$$
I(y):=\left\{k\in\{1,\dots, n\}\,:\,\frac{1}{|y_k^*|\sqrt{\frac{1}{k}\sum_{i=1}^k\frac{1}{|y_i^*|^2}}} \left(\leq\frac{|y_1^*|}{|y_k^*|}\right)\leq n^{c_1}\right\},
$$
where $y^*$ denotes the decreasing rearrangement of $y$, then
$$
\E w(K_{N,y})\geq \sup_{k\in I(y)}\frac{c_2\sqrt{\log (k+1)}}{\sqrt{\frac{1}{k}\sum_{i=1}^k\frac{1}{|y_k^*|^2}}} \,.
$$
\end{proposition}

\begin{rmk}
Notice that in the case $y=(1,\dots, 1)$ and $n\leq N\leq n^\delta$ we recover the exact lower bound for $\E w(K_N)$.
\end{rmk}
\begin{proof}
Let $y\in \R^N$ and $k\in I(y)$. We can assume without loss of generality that $y=y^*$. Obviously, we have that
$$
K_{N,y}\supseteq \textrm{conv}\{\pm y_1X_1,\dots,\pm y_kX_k\}=: K_{k,y}.
$$
Thus, for every $\theta\in S^{n-1}$
$$
h_{K_{N,y}}(\theta)\geq h_{K_{k,y}}(\theta)=\max_{1\leq i\leq k}|\langle y_iX_i,\theta\rangle|,
$$
and then
$$
\E h_{K_{N,y}}(\theta)\geq \E h_{K_{k,y}}(\theta)=\E \max_{1\leq i\leq k}|\langle y_iX_i,\theta\rangle|.
$$
By Theorem \ref{THM Schuett Werner Litvak Gordon}
$$
\E \max_{1\leq i\leq k}|\langle y_iX_i,\theta\rangle|  \sim \left\|y\right\|_{M_\theta} = \inf\left\{s>0\,:\,\sum_{i=1}^k M_\theta\left(\frac{|y_i|}{s}\right)\leq 1\right\} .
$$
Moreover, $M_\theta\left(\frac{|y_i|}{s}\right)\geq \Pro\left( |\langle X,\theta\rangle|\geq \frac{2s}{|y_i|}\right)$. This holds because for every $i=1,\ldots,k$
\begin{eqnarray*}
M_\theta\left(\frac{|y_i|}{s}\right) & = & \int_{0}^{\frac{|y_i|}{s}} \int_{\{|\langle x,\theta\rangle| \geq \frac 1 t \}} |\langle x,\theta \rangle| \,d\Pro\,dt \\
& \geq & \int_{\frac{|y_i|}{2s}}^{\frac{|y_i|}{s}} \int_{\{|\langle x,\theta\rangle| \geq \frac 1 t \}} |\langle x,\theta \rangle| \,d\Pro\,dt \\
& \geq & \frac{2s}{|y_i|} \int_{\frac{|y_i|}{2s}}^{\frac{|y_i|}{s}} \Pro\left( |\langle X,\theta\rangle| \geq \frac{2s}{|y_i|} \right) \,dt \\
& = & \Pro\left(|\langle X,\theta\rangle| \geq \frac{2s}{|y_i|} \right).
\end{eqnarray*}
Hence, we obtain for every $\theta\in S^{n-1}$
$$
\E h_{K_{N,y}}(\theta) \geq c \inf \left\{ s>0 \,:\, \sum_{i=1}^k \Pro\{|\langle X,\theta\rangle| \geq \frac{2s}{|y_i|} \} \leq 1\right\}.
$$
Now, if for some $s_0>0$ we have
$$
  \sum_{i=1}^k \Pro\left( |\langle x,\theta\rangle|\geq \frac{2s_0}{|y_i|}\right)|>1,
$$
we obtain $\E h_{K_{N,y}}(\theta)\geq cs_0$.

By Theorem \ref{THM_CLT}, there is a constant $\kappa$ and a set $\Theta\subseteq S^{n-1}$ with $\sigma(\Theta)\geq 1-Ce^{-\sqrt n}$ such that if $\theta\in\Theta$ and $\frac{2t}{|y_i|}\leq n^\kappa$
\begin{eqnarray*}
\Pro \left ( |\langle X,\theta\rangle| \geq \frac{2t}{|y_i|}\right) & > & 1-\left(1+\frac{C^\prime}{n^\kappa}\right)\left(1-\frac{|y_i|e^{-\frac{2t^2}{|y_i|^2}}}{2t\sqrt{2\pi}}\right)\cr
&\geq&\frac{|y_i|e^{-\frac{2t^2}{|y_i|^2}}}{2t\sqrt{2\pi}}-\frac{C^\prime}{n^\kappa}.
\end{eqnarray*}
Thus, we are looking for $t$ such that
$$
\sum_{i=1}^k\frac{|y_i|e^{-\frac{2t^2}{|y_i|^2}}}{2t\sqrt{2\pi}}\geq2+ C^\prime \frac{k}{n^\kappa},
$$
and, therefore, a $t$ such that
$$
\frac{1}{k}\sum_{i=1}^n\frac{|y_i|e^{-\frac{2t^2}{|y_i|^2}}}{2t\sqrt{2\pi}}\geq \frac{C^{\prime\prime}}{\min\{k, n^{\kappa}\}}
$$
works. By the arithmetic-geometric mean inequality
$$
\frac{1}{k}\sum_{i=1}^k\frac{|y_i|e^{-\frac{2t^2}{|y_i|^2}}}{2t\sqrt{2\pi}}\geq\frac{1}{2t\sqrt{2\pi}}\prod_{i=1}^k|y_i|^\frac{1}{k}e^{-\frac{2t^2}{k}\sum_{i=1}^k\frac{1}{|y_i|^2}}.
$$
Taking $t=\frac{\sqrt{\alpha\log (k+1)}}{\sqrt 2\sqrt{\frac{1}{k}\sum_{i=1}^k\frac{1}{|y_i|^2}}}$ we obtain that this quantity equals
$$
\frac{\sqrt{\frac{1}{k}\sum_{i=1}^k\frac{1}{|y_i|^2}}\prod_{i=1}^k|y_i|^\frac{1}{k}}{2\sqrt\pi\sqrt{\alpha\log (k+1)}k^\alpha}\geq\frac{1}{2\sqrt\pi\sqrt{\alpha\log (k+1)}k^\alpha},
$$
which is greater than $\frac{C^{\prime\prime}}{\min\{k,n^\kappa\}}$ if $\alpha$ is a constant small enough. Thus, for every $\theta\in \Theta$,
$$
\E h_{K_{N,y}}\geq \frac{c\sqrt{\log (k+1)}}{\sqrt{\frac{1}{k}\sum_{i=1}^k\frac{1}{|y_i^*|^2}}}.
$$
By Markov's inequality
$$
\E w(K_{N,y})\geq \frac{c_2\sqrt{\log (k+1)}}{\sqrt{\frac{1}{k}\sum_{i=1}^k\frac{1}{|y_i^*|^2}}}.
$$
\end{proof}

\proof[Acknowledgements]
This work was done while the first named author visited the second named author at Johannes Kepler University in Linz, Austria.
We would like to thank the department for the financial support and for providing such good environment and working conditions. We would also like to thank our colleague
Nikos Dafnis for helpful discussions.

\end{document}